\documentclass[11pt,onecolumn,oneside]{article}
\usepackage[a4paper, hmargin={2.8cm, 2.8cm}, vmargin={2.5cm, 2.5cm}]{geometry}

\usepackage[utf8]{inputenc}      		 
\usepackage[T1]{fontenc}                 
\usepackage{graphicx, color}             
\usepackage{textcomp}                    
\usepackage{amssymb}
\usepackage{amsmath}
\usepackage{amsthm}						 
\usepackage{mathrsfs}
\usepackage{caption, subcaption}
\usepackage{wrapfig}
\usepackage{gensymb}
\usepackage{mathtools}
\usepackage{bm} 
\usepackage{subfiles}
\usepackage[normalem]{ulem}              

\DeclareFontFamily{U}{mathx}{\hyphenchar\font45}
\DeclareFontShape{U}{mathx}{m}{n}{
      <5> <6> <7> <8> <9> <10>
      <10.95> <12> <14.4> <17.28> <20.74> <24.88>
      mathx10
      }{}
\DeclareSymbolFont{mathx}{U}{mathx}{m}{n}
\DeclareFontSubstitution{U}{mathx}{m}{n}
\DeclareMathAccent{\widecheck}{0}{mathx}{"71}
\DeclareMathAccent{\wideparen}{0}{mathx}{"75}

\usepackage{tikz}
\usetikzlibrary{cd}
\usepackage{verbatim}

\usepackage{enumerate}	      	
\usepackage{hyperref}			
\usepackage{xcolor}				
\hypersetup{					
    colorlinks,
    linkcolor={red!50!black},
    citecolor={green!30!black},
    urlcolor={blue!80!black},
}
\usepackage{xkvltxp} 
\usepackage[draft,nomargin,inline,index]{fixme} 
\fxsetup{theme=color}

\newcommand{\numberthis}{\addtocounter{equation}{1}\tag{\theequation}}

\newcommand{\N}{\ensuremath{\mathbb{N}}}

\newcommand{\C}{\ensuremath{\mathbb{C}}}

\newcommand{\T}{\ensuremath{\mathbb{T}}}

\newcommand{\restrict}[1]{\raisebox{-0.15ex}{$|$}_{#1}}
\newcommand{\integral}[3]{\int_{#2} #1\,\mathrm{d}#3}

\DeclareMathOperator{\aut}{Aut} 
\DeclareMathOperator{\Iso}{Iso} 

\newcommand{\act}[1]{\overset{#1}{\curvearrowright}}

\newcommand{\set}[2]{\left\{\,#1\;\middle|\; #2\,\right\}}

\makeatletter
\def\moverlay{\mathpalette\mov@rlay}
\def\mov@rlay#1#2{\leavevmode\vtop{%
   \baselineskip\z@skip \lineskiplimit-\maxdimen
   \ialign{\hfil$\m@th#1##$\hfil\cr#2\crcr}}}
\newcommand{\charfusion}[3][\mathord]{
    #1{\ifx#1\mathop\vphantom{#2}\fi
        \mathpalette\mov@rlay{#2\cr#3}
      }
    \ifx#1\mathop\expandafter\displaylimits\fi}
\makeatother

\newcommand{\seq}[2]{\left(#1_{#2}\right)_{#2\geq1}}

\newcommand{\net}[3]{\left(#1_{#2}\right)_{#2\in#3}}

\newcommand{\norm}[2]{\ensuremath{\left|\!\left|#1\right|\!\right|_{#2}}}

\newcommand{\abs}[1]{\ensuremath{\left|#1\right|}}

\usepackage{etoolbox}
\newcommand{\addQEDstyle}[2]{\AtBeginEnvironment{#1}{\pushQED{\qed}\renewcommand{\qedsymbol}{#2}}\AtEndEnvironment{#1}{\popQED}}

\theoremstyle{plain}
\newtheorem{thm}{Theorem}
\numberwithin{thm}{section}
\newtheorem{lem}[thm]{Lemma}
\newtheorem{cor}[thm]{Corollary}
\newtheorem{prop}[thm]{Proposition}

\theoremstyle{definition}

\addQEDstyle{ex}{$\circ$} 

\newtheoremstyle{question}
{\topsep}
{\topsep}
{}
{}
{\bfseries}
{.}
{.5em}
{}
\theoremstyle{question}
\newtheorem*{question}{Question}

\theoremstyle{remark}
\newtheorem{rk}[thm]{Remark}

\theoremstyle{plain}
\newtheorem{introthm}{Theorem}

\theoremstyle{definition}
\newtheorem{introdefn}[introthm]{Definition}

\title{Weak property $(\mathrm{T}_{L^p})$ for discrete groups}
\author{Emilie Mai Elki\ae{}r}
\date{\today}

\begin{document}

\maketitle

\begin{abstract}
We show that, for a countable discrete group $\Gamma$, property $(\mathrm{T}_{L^p})$ of Bader, Furman, Gelander and Monod is equivalent to the property that, whenever an $L^p$-representation of $\Gamma$ admits a net of almost invariant unit vectors, it has a non-zero invariant vector. Central in the proof is to show that the closure of the group of $\T$-valued $1$-coboundaries is a sufficient criteria for strong ergodicity of ergodic p.m.p. actions.
\end{abstract}

\section{Introduction}
Kazhdan's property $(\rm T)$ is a rigidity property concerning how a group may act on a Hilbert space. It was first introduced by Kazhdan in 1967 in \cite{Kazhdan1967} and has since then become an important notion in analytic group theory. We refer to \cite{BekkaDeLaHarpeValette} for a thorough introduction to the topic. In their seminal paper \cite{BaderFurmanGelanderMonod} from 2007, Bader, Furman, Gelander and Monod brought the notion of property $(\rm T)$ to the broader framework of Banach spaces. Their property $(\mathrm{T}_{\mathcal{E}})$, where $\mathcal{E}$ is a class of Banach spaces, is a rigidity property concerning how a group may act on spaces in the class $\mathcal{E}$. Since then, many authors have studied rigidity for actions on Banach spaces -- see, e.g., \cite{BekkaOlivier2014Tlp}, \cite{delaSalle2015TowardsStrongBanachT}, \cite{deLaatdelaSalle2021SpectralGap}, \cite{LavyOlivier2014FixedpointSF}, \cite{MarrakchiDeLASalle2023DependenceOnp}, \cite{Oppenheim2023BanachTforSLnZ} and \cite{Oppenheim2023Zuk}, to name a few.

Let $\Gamma$ be a discrete group. For a Banach space $E$, we denote by $\Iso(E)$ the group of linear surjective isometries on $E$. An isometric representation $(\pi,E)$ of $\Gamma$ on $E$ is a group homomorphism $\pi:\Gamma\rightarrow\Iso(E)$. When $E$ is a Hilbert space, an isometric representation is commonly known as a unitary representation. Given an isometric representation $(\pi,E)$, a vector $\xi\in E$ is said to be \emph{invariant} if $\pi(t)\xi=\xi$, for all $t\in\Gamma$. It is easy to verify that the set of invariant vectors forms a subspace of $E$, which we denote $E^{\pi}$. A net of vectors $\net{\xi}{i}{I}$ is said to be \emph{almost invariant} if $\norm{\pi(t)\xi_i-\xi_i}{E}\rightarrow 0$, for every $t\in\Gamma$.
\begin{introdefn}[Property $(\mathrm{T}_{\mathcal{E}})$]\label{def:T_E}
Let $\mathcal{E}$ be a class of Banach spaces. A discrete group $\Gamma$ has \emph{property $(\mathrm{T}_{\mathcal{E}})$} if, whenever an isometric representation $(\pi,E)$ of $\Gamma$ with $E$ in the class $\mathcal{E}$ admits a net of almost invariant unit vectors $\net{\xi}{i}{I}$, there exists a net of invariant vectors $\net{\eta}{i}{I}$ such that $\norm{\xi_i-\eta_i}{E}\rightarrow0$.
\end{introdefn}
In \cite{BaderFurmanGelanderMonod}, property $(\mathrm{T}_{\mathcal{E}})$ is defined in terms of the lack of nets of almost invariant unit vectors in the quotient $E/E^{\pi}$. By \cite[Lemma 18]{Tanaka2017PropertyT}, Definition \ref{def:T_E} above is equivalent to the definition of Bader, Furman, Gelander and Monod. When $\mathcal{E}$ is the class of complex Hilbert spaces, we recover Kazhdan's property $(\rm T)$.

In the classical setting of unitary representations on Hilbert spaces, it is well-known that Kazhdan's property $(\rm T)$ allows several equivalent formulations. In particular, property $(\rm T)$ is often defined as the property that the existence of almost invariant unit vectors forces the existence of a non-zero invariant vector. When generalizing this property to the setting of actions on Banach spaces, we obtain an a priori weaker version of property $(\mathrm {T}_{\mathcal{E}})$.
\begin{introdefn}[Weak property $(\mathrm{T}_{\mathcal{E}})$]\label{def:weakT_E}
Let $\mathcal{E}$ be a class of Banach spaces. A discrete group $\Gamma$ has \emph{weak property $(\mathrm {T}_{\mathcal{E}})$} if any isometric representation $(\pi,E)$ of $\Gamma$ with $E$ in the class $\mathcal{E}$ admitting a net of almost invariant unit vectors has a non-zero invariant vector.
\end{introdefn}
While it is easy to see that weak property $(\mathrm{T}_{\mathcal{E}})$ is implied by property $(\mathrm{T}_{\mathcal{E}})$, the converse implication depends on the class $\mathcal{E}$. It is well-known to experts that the following two conditions are sufficient to ensure the equivalence of weak property $(\mathrm{T}_{\mathcal{E}})$ and property $(\mathrm{T}_{\mathcal{E}})$:
\begin{enumerate}[(i)]
    \item $\mathcal{E}$ is stable under quotients,
    \item $\mathcal{E}$ is a class of superreflexive Banach spaces stable under taking complemented subspaces.
\end{enumerate}
A proof of this can be found in \cite[Proposition 2.20]{ElkiaerPooya2023}. Each of these conditions cover the case where $\mathcal{E}$ is the class of Hilbert spaces. But interestingly, the class $L^p$ of $L^p$-spaces on $\sigma$-finite measure spaces, for $1\leq p<\infty$, does not satisfy either of the two conditions unless $p=2$. In this paper, we address the question if weak property $(\mathrm{T}_{L^p})$ is the same as property $(\mathrm{T}_{L^p})$. We show that the answer is affirmative when the group in question is countable discrete.
\begin{introthm}[Theorem \ref{thm:weakTLp_and_TLp}]\label{introthm:weakTLp_and_TLp}
A countable discrete group has property $(\mathrm{T}_{L^p})$ if and only if it has weak property $(\mathrm{T}_{L^p})$.
\end{introthm}

The proof of Theorem \ref{introthm:weakTLp_and_TLp} relies on an analysis of ergodicity of measure preserving actions on probability spaces (in short: p.m.p. actions). The connection to property $(\rm T)$ is given by the characterization by Connes and Weiss in \cite{ConnesWeiss1980}: A discrete group has property $(\rm T)$ if and only if every p.m.p. ergodic action (on a diffuse standard probability space) is strongly ergodic. We establish, via an application of the open mapping theorem for Polish groups, that the closure in a natural topology of the group of $\T$-valued $1$-coboundaries for a given ergodic p.m.p. action is a sufficient condition to ensure strong ergodicity. This may be of independent interest.

\begin{introthm}[Theorem \ref{thm:B1_closed_implies_strongly_ergodic}]\label{introthm:B1_closed_implies_strongly_ergodic}
Let $\Gamma$ be a countable discrete group, $(\Omega,\nu)$ a separable probability space and $\Gamma\act{\sigma}(\Omega,\nu)$ an ergodic p.m.p. action. If $B^1(\sigma;\T)$ is closed then $\sigma$ is strongly ergodic.
\end{introthm}

A fair point can be made that the sufficient conditions (i) and (ii) above for equivalence of weak property $(\mathrm{T}_{\mathcal{E}})$ and property $(\mathrm{T}_{\mathcal{E}})$ are stronger than needed. For example, it is not necessary to require that every quotient stays in the class $\mathcal{E}$. Instead, given an isometric representation $(\pi,L^p(\Omega,\nu))$, we are interested only in the quotient $L^p(\Omega,\nu)/L^p(\Omega,\nu)^{\pi}$. This raises the question if a proof of Theorem \ref{introthm:weakTLp_and_TLp} is viable without appealing to Theorem \ref{introthm:B1_closed_implies_strongly_ergodic}. We give a partial negative answer to this question in Theorem \ref{introthm:Lp0_is_not_an_Lp-space}. When $\pi$ comes from an ergodic p.m.p. action, the quotient $L^p(\Omega,\nu)/L^p(\Omega,\nu)^{\pi}$ is equivariantly and isometrically isomorphic to the dual of the subspace $L^{p'}_0(\Omega,\nu)$ of functions with mean zero, where $p'$ is the Hölder conjugate of $p$. We show that the subspace $L^p_0(\Omega,\nu)$ in many cases is not isometrically isomorphic to an $L^p$-space on a $\sigma$-finite measure space, and so, neither is its dual.

\begin{introthm}[Theorem \ref{thm:Lp0_is_not_an_Lp-space}]\label{introthm:Lp0_is_not_an_Lp-space}
Let $(\Omega,\nu)$ be a diffuse standard probability space and let $1\leq p<\infty$, $p\not\in2\N$. Then $L^p_0(\Omega,\nu)$ is not isometrically isomorphic to an $L^p$-space on a $\sigma$-finite measure space.
\end{introthm}

This paper is organized as follows: In section \ref{sec:pre} we cover the preliminaries on Polish groups, actions on measure spaces, ergodicity and strong ergodicity, and on the topological groups of $\T$-valued $1$-cocycles and $1$-coboundaries. In section \ref{sec:erg_and_1-cobound} we prove Theorem \ref{introthm:B1_closed_implies_strongly_ergodic} and in section \ref{sec:weakTLp} we prove Theorem \ref{introthm:weakTLp_and_TLp}. Finally, in section \ref{sec:no_easier_proof}, we prove Theorem \ref{introthm:Lp0_is_not_an_Lp-space}.

\paragraph{Acknowledgements} The author is grateful to Mikael de la Salle for the question and for discussions leading to the proof presented in this paper. The author thanks Todor Tsankov for generously sharing his insights into Polish groups. The author thanks Nadia Larsen for comments on an earlier version of this paper leading to improvements in the presentation. The author thanks the Trond Mohn foundation for supporting a research visit in Lyon where this project was initiated.



\section{Preliminaries}\label{sec:pre}
\paragraph{Polish groups}\label{pre-par:polish} A topological space is said to be \emph{Polish} if it is separable and completely metrizable. A topological group is said to be \emph{Polish} if it is Polish as a topological space. We list here a few permanence properties for Polish groups that shall become useful to us later. See, e.g., \cite[Section 3]{Kechris-DescriptiveSetTheory} for a reference.

\begin{prop}\label{prop:Polish_permanence_subgroup}
A closed subgroup of a Polish group is Polish.
\end{prop}

\begin{prop}\label{prop:Polish_permanence_countable_product}
A countable product of Polish groups is a Polish group.
\end{prop}

A main tool in this paper is the Open Mapping Theorem in the setting of Polish groups, which we state below in Theorem \ref{thm:open_mapping_Polish}. It follows directly from Effros' Theorem \cite[Theorem 2.1]{Effros1965} (see also \cite{Ancel1987} and \cite[Theorem 2.2.2]{BeckerKechris1996TheDSofPolish}).
\begin{thm}[Open Mapping Theorem for Polish groups]\label{thm:open_mapping_Polish}
Let $G$ and $H$ be a Polish groups and let $\Phi:G\rightarrow H$ be a continuous and surjective group homomorphism. Then $\Phi$ is open.
\end{thm}

\paragraph{Actions on measure spaces}
An introduction to group actions on measure spaces can be found, e.g., in \cite[Section A.6]{BekkaDeLaHarpeValette}. We recall briefly the main definitions. Let $\Gamma$ be a discrete group. Given a $\sigma$-finite measure space $(\Omega,\nu)$, we denote by $\aut(\Omega,[\nu])$ the group of all bi-measurable transformations of $\Omega$ that leave $\nu$ quasi-invariant. A \emph{measure class preserving action} of $\Gamma$ on $(\Omega,\nu)$ is a group homomorphism $\sigma:\Gamma\rightarrow\aut(\Omega,[\nu])$. We write $\Gamma\act{\sigma}(\Omega,\nu)$ for the action of $\Gamma$ on $(\Omega,\nu)$ given by $\sigma$. For $t\in\Gamma$ and $\omega\in\Omega$, we shall often write $t.\omega$ instead of $\sigma_t(\omega)$. Further, we denote by $t.\nu$ the push forward measure of $\nu$ by $\sigma_t$. For each $t\in\Gamma$, we denote by $\tfrac{\mathrm{d}t.\nu}{\mathrm{d}\nu}$ the Radon-Nikodym derivative of $t.\nu$ with respect to $\nu$. The assumption that the action is measure class preserving ensures that this Radon-Nikodym derivative exists. Recall that it is a strictly positive function. Let $L^0(\Omega,\nu)$ denote the space of (equivalence classes of) measurable complex-valued functions on $(\Omega,\nu)$. For $\xi\in L^0(\Omega,\nu)$ and $t\in\Gamma$, we denote by $t.\xi$ be the measurable function given by $t.\xi(\omega)=\xi(t^{-1}.\omega)$, for $\omega\in\Omega$. In this way, the action $\Gamma\act{\sigma}(\Omega,\nu)$ induces in a canonical way an action of $\Gamma$ on $L^0(\Omega,\nu)$.

When an action $\Gamma\act{\sigma}(\Omega,\nu)$ leaves the measure $\nu$ invariant rather than just quasi-invariant, we say that it is \emph{measure preserving}. A measure preserving action on a probability space is called a  \emph{probability measure preserving action} (in short: a p.m.p. action). If the action is measure preserving, $\tfrac{\mathrm{d}t.\nu}{\mathrm{d}\nu}$ is everywhere equal to $1$, for all $t\in\Gamma$.

\paragraph{Ergodicity and strong ergodicity}
Let $\Gamma\act{\sigma}(\Omega,\nu)$ be a measure class preserving action of a discrete group on a $\sigma$-finite measure space. A measurable subset $A$ of $\Omega$ is said to be \emph{$\Gamma$-invariant} if $\nu(t.A\triangle A)=0$, for all $t\in\Gamma$. Observe that null and co-null subsets are trivially $\Gamma$-invariant. We say that the action $\Gamma\act{\sigma}(\Omega,\nu)$ is \emph{ergodic} if there are no non-trivial $\Gamma$-invariant measurable subsets of $\Omega$.

When $\Gamma\act{\sigma}(\Omega,\nu)$ is a p.m.p. action, a stronger version of ergodicity is defined as follows: A sequence $\seq{A}{n}$ of measurable subsets of $\Omega$ is said to be \emph{asymptotically $\Gamma$-invariant} if $\nu(t.A_n\triangle A_n)$ converges to zero, for all $t\in\Gamma$. The sequence is trivially asymptotically $\Gamma$-invariant if $\liminf_n\nu(A_n)\nu(A_n^{\complement})=0$. The action $\Gamma\act{\sigma}(\Omega,\nu)$ is said to be \emph{strongly ergodic} if there are no non-trivial asymptotically $\Gamma$-invariant sequences of measurable subsets of $\Omega$. It is clear that any strongly ergodic action is automatically ergodic. A deep result by Connes and Weiss in \cite{ConnesWeiss1980} (see also \cite[Theorem 6.3.4]{BekkaDeLaHarpeValette}) shows that the converse implication characterizes groups with property $(\rm T)$. We state it here for discrete groups.
\begin{thm}[Connes--Weiss]\label{thm:ConnesWeiss}
A discrete group $\Gamma$ has property $(\rm T)$ if and only if every ergodic p.m.p. action is strongly ergodic.
\end{thm}
\begin{rk}\label{rk:ConnesWeiss}
In Theorem \ref{thm:ConnesWeiss}, it is enough to consider actions on diffuse standard probability spaces, i.e., probability spaces that are isomorphic mod $0$ to the interval with the Lebesgue measure. This follows from the proof of Corollary A.7.15 in \cite{BekkaDeLaHarpeValette} and Theorem 17.41 in \cite{Kechris-DescriptiveSetTheory}.
\end{rk}
We refer to \cite[Section 6.3]{BekkaDeLaHarpeValette} for more background on ergodicity of group actions and a comprehensive review of the connection between ergodicity and property $(\rm T)$.

\paragraph{The $\T$-valued measurable functions as a topological group}
Let $(\Omega,\nu)$ be a probability space. We denote by $L^0(\Omega,\nu;\T)$ the set of measurable functions on $(\Omega,\nu)$ with values in $\T$. We have a natural group structure on $L^0(\Omega,\nu;\T)$ with multiplication defined pointwise. The multiplicative unit is the function $1_{\Omega}$ which is everywhere equal to $1$. The inverse of a function in $L^0(\Omega,\nu;\T)$ is its complex conjugate. We equip $L^0(\Omega,\nu;\T)$ with the topology of convergence in measure, i.e., the topology generated by the sets of the form
\[
V_{\varepsilon}(\varphi_0)=\set{\varphi\in L^0(\Omega,\nu;\T)}{\nu(\{\abs{\varphi-\varphi_0}\geq\varepsilon\})<\varepsilon},\label{defeq:top_on_L^0}
\]
where $\varphi_0\in L^0(\Omega,\nu;\T)$ and $\varepsilon>0$. A standard computation verifies that this makes $L^0(\Omega,\nu;\T)$ a topological group. Lemma \ref{lem:top_on_L0(Omega,nu;D)} below is well-known to experts. It shows, in particular, that the topology on $L^0(\Omega,\nu;\T)$ is metrizable. 
\begin{lem}\label{lem:top_on_L0(Omega,nu;D)}
Let $(\Omega,\nu)$ be a probability space and fix $1\leq p <\infty$. The topology on $L^0(\Omega,\nu;\T)$ generated by the sets in (\ref{defeq:top_on_L^0}) is equivalent to the topology generated by the sets
\[
V^p_{\varepsilon}(\varphi_0)=\set{\varphi\in L^0(\Omega,\nu;\T)}{\norm{\varphi-\varphi_0}{p}<\varepsilon},
\]
where $\varphi_0\in L^0(\Omega,\nu;\T)$ and $\varepsilon>0$.
\end{lem}
\begin{proof}
Denote by $\mathcal{T}_0$ the topology on $L^0(\Omega,\nu;\T)$ generated by the sets of the form $V_{\varepsilon}(\varphi_0)$ and by $\mathcal{T}_p$ the topology generated by the sets of the form $V^p_{\varepsilon}(\varphi_0)$. Fix $\varphi_0\in L^0(\Omega,\nu;\T)$ and $\varepsilon>0$. If $\varphi\in L^0(\Omega,\nu;\T)$ is such that $\norm{\varphi-\varphi_0}{p}^p<\varepsilon^{p+1}$ then
\[
\nu\left(\{\abs{\varphi-\varphi_0}\geq\varepsilon\}\right) 
\leq \frac{1}{\varepsilon^p}\integral{\abs{\varphi-\varphi_0}^p}{\{\abs{\varphi-\varphi_0}\geq\varepsilon\}}{\nu}
\leq \frac{1}{\varepsilon^p}\integral{\abs{\varphi-\varphi_0}^p}{\Omega}{\nu}<\varepsilon.
\]
Hence, $V^p_{\varepsilon^{1+1/p}}(\varphi_0)\subset V_{\varepsilon}(\varphi_0)$. This shows that $\mathcal{T}_p$ is finer than $\mathcal{T}_0$. Conversely, if $\varphi\in L^0(\Omega,\nu;\T)$ is such that $\nu(\{\abs{\varphi-\varphi_0}\geq\varepsilon\})<\varepsilon$ then
\begin{align*}
\integral{\abs{\varphi-\varphi_0}^p}{\Omega}{\nu}&=\integral{\abs{\varphi-\varphi_0}^p}{\{\abs{\varphi-\varphi_0}\geq\varepsilon\}}{\nu}+\integral{\abs{\varphi-\varphi_0}^p}{\{\abs{\varphi-\varphi_0}<\varepsilon\}}{\nu}\\
&\leq 2^p\,\nu\left(\{\abs{\varphi-\varphi_0}\geq\varepsilon\}\right)+\varepsilon^p\,\nu\left(\{\abs{\varphi-\varphi_0}<\varepsilon\}\right)\\
&<\varepsilon2^p+\varepsilon^p.
\end{align*}
Hence, $V_{\varepsilon}(\varphi_0)\subset V^p_{(\varepsilon2^p+\varepsilon^p)^{1/p}}(\varphi_0)$. Since $(\varepsilon2^p+\varepsilon^p)^{1/p}\rightarrow0$ as $\varepsilon\rightarrow0$, this shows that $\mathcal{T}_0$ is finer than $\mathcal{T}_p$. Hence, the two topologies are equivalent.
\end{proof}

Recall that a measure space $(\Omega,\nu)$ is said to be separable if the space of measurable subsets of $\Omega$ is separable as a topological space with respect to the distance given by $\nu(A\triangle B)$, for $A,B\subset\Omega$ measurable.

\begin{prop}\label{prop:L^0(Omega,nu)_is_Polish}
Let $(\Omega,\nu)$ be a separable probability space. Then $L^0(\Omega,\nu;\T)$ is a Polish group with multiplication defined pointwise and with the topology generated by the sets defined in equation (\ref{defeq:top_on_L^0}).
\end{prop}
\begin{proof}
Because $\nu$ is finite, any measurable function with values in $\T$ is integrable. Hence, we can view $L^0(\Omega,\nu;\T)$ as a subset of $L^1(\Omega,\nu)$. By Lemma \ref{lem:top_on_L0(Omega,nu;D)}, the subspace topology on $L^0(\Omega,\nu;\T)$ coming from this embedding agrees with the topology generated by the sets defined in equation (\ref{defeq:top_on_L^0}). It is a standard computation to verify that $L^0(\Omega,\nu;\T)$ is closed in $L^1(\Omega,\nu)$. Hence, $L^0(\Omega,\nu;\T)$ is completely metrizable. Finally, as $(\Omega,\nu)$ is separable, $L^1(\Omega,\nu)$ is separable and then so is $L^0(\Omega,\nu;\T)$.
\end{proof}

\paragraph{Groups of $1$-cocycles and $1$-coboundaries}
Let $\Gamma$ be a discrete group and $\Gamma\act{\sigma}(\Omega,\nu)$ a measure class preserving action. A ($\T$-valued) \emph{$1$-cocycle} for the action $\Gamma\act{\sigma}(\Omega,\nu)$ is a map $c:\Gamma\times\Omega\rightarrow\T$ such that $c_t$ is a measurable map, for every $t\in\Gamma$, and such that, for every pair $s,t\in \Gamma$ and almost every $\omega\in\Omega$,
\[
c_{st}(\omega)=c_s(\omega)c_t(s^{-1}.\omega).
\]
The set of all $1$-cocycles is denoted by $Z^1(\sigma;\T)$. Given $\varphi\in L^0(\Omega,\nu;\T)$, define a map $b_{\varphi}:\Gamma\times\Omega\rightarrow\T$ by
\[
b_{\varphi}(t,\omega)=\frac{\varphi(\omega)}{\varphi(t^{-1}.\omega)},\numberthis\label{eqn:def:coboundary}
\]
for $t\in\Gamma$ and $\omega\in\Omega$. It is straight forward to verify that $b_{\varphi}$ is a $1$-cocycle. A $1$-cocycle of this form is called a \emph{$1$-coboundary}. The set of all $1$-coboundaries is denoted by $B^1(\sigma;\T)$.

We equip $Z^1(\sigma;\T)$ with a group structure as follows: Given two $1$-cocycles $c,d\in Z^1(\sigma;\T)$, their product is defined via the multiplication on $\T$ by setting $(c\cdot d)_t(\omega)=c_t(\omega)d_t(\omega)$, for $t\in\Gamma$ and $\omega\in\Omega$. Because $\T$ is an abelian group, $c\cdot d$ is again a $1$-cocycle, and so, $\cdot$ gives a well-defined multiplication on $Z^1(\sigma;\T)$. The $1$-coboundary $b_{1_{\Omega}}$ is the multiplicative identity. The inverse of a $1$-cocycle is given by its complex conjugate. Further, we have a canonical embedding of groups $Z^1(\sigma;\T)\hookrightarrow L^0(\Omega,\nu;\T)^{\Gamma}$ given by mapping $c\in Z^1(\sigma;\T)$ to the $\Gamma$-indexed sequence $(c_t)_{t\in\Gamma}$. This embedding gives $Z^1(\sigma;\T)$ the structure of a topological group inheriting the product topology from $L^0(\Omega,\nu;\T)^{\Gamma}$. It can be verified with a standard computation that, as such, it is closed. We include a proof in Proposition \ref{prop:Z^1_is_Polish} below that $Z^1(\sigma;\T)$ is Polish under the additional assumptions that $\Gamma$ is countable and $(\Omega,\nu)$ is separable. This fact can be found without proof in \cite[Section 24]{Kechris2010GlobalAspErgGpAct}.
\begin{prop}\label{prop:Z^1_is_Polish}
Let $\Gamma$ be a countable discrete group and let $\Gamma\act{\sigma}(\Omega,\nu)$ be a measure class preserving action on a separable probability space. Then $Z^1(\sigma;\T)$ is a Polish group.
\end{prop}
\begin{proof}
When $\Gamma$ is countable discrete and $(\Omega,\nu)$ is a separable probability space, $L^0(\Omega,\nu;\T)^{\Gamma}$ is a Polish group by Proposition \ref{prop:L^0(Omega,nu)_is_Polish} and Proposition \ref{prop:Polish_permanence_countable_product}. Because $Z^1(\sigma;\T)$ is closed in $L^0(\Omega,\nu;\T)^{\Gamma}$, it follows from Proposition \ref{prop:Polish_permanence_subgroup} that $Z^1(\sigma;\T)$ is a Polish group.
%
\end{proof}
We equip $B^1(\sigma;\T)$ with the subspace topology. When $B^1(\sigma;\T)$ is closed, it is Polish. But this need not be the case. In particular, we shall see in the proof of Theorem \ref{introthm:B1_closed_implies_strongly_ergodic} that $B^1(\sigma;\T)$ is not Polish when $\sigma$ is an ergodic but not strongly ergodic p.m.p. action. By the characterization of property $(\rm T)$ by Connes and Weiss (Theorem \ref{thm:ConnesWeiss}), every group without property $(\rm T)$ admits such an action.

For more background on $1$-cocycles and $1$-coboundaries, we refer to \cite[Chapter 3]{Kechris2010GlobalAspErgGpAct}.

\section{Proof of Theorem \ref{introthm:B1_closed_implies_strongly_ergodic}}\label{sec:erg_and_1-cobound}
In this section, we prove Theorem \ref{introthm:B1_closed_implies_strongly_ergodic} from the introduction (Theorem \ref{thm:B1_closed_implies_strongly_ergodic} below), which gives a connection between the closure of the space of $\T$-valued $1$-coboundaries and strong ergodicity of the action. The main tool in the proof is an application of the Open Mapping Theorem for Polish groups (Theorem \ref{thm:open_mapping_Polish}) to the map $\beta:L^0(\Omega,\nu;\T)\rightarrow B^1(\sigma;\T)$ given by $\beta(\varphi)=b_{\varphi}$, where $b_{\varphi}$ is as defined equation (\ref{eqn:def:coboundary}).

\begin{lem}\label{lem:beta:L0->B1}
Let $\Gamma$ be a discrete group and let $\Gamma\act{\sigma}(\Omega,\nu)$ be an ergodic p.m.p. action. For each $\varphi\in L^0(\Omega,\nu;\T)$, let $b_{\varphi}\in B^1(\sigma;\T)$ be as in (\ref{eqn:def:coboundary}). The map $\beta:L^0(\Omega,\nu;\T)\rightarrow B^1(\sigma;\T)$ given by $\beta(\varphi)=b_{\varphi}$ is a continuous and surjective group homomorphism whose kernel is the subgroup of constant functions.
\end{lem}
\begin{proof}
For $\varphi,\psi\in L^0(\Omega,\nu;\T)$, $t\in \Gamma$ and $\omega\in\Omega$, we have
\[
b_{\varphi\cdot\psi}(t,\omega)=\frac{\varphi(\omega)\psi(\omega)}{\varphi(t^{-1}.\omega)\psi(t^{-1}.\omega)}=b_{\varphi}(t,\omega)b_{\psi}(t,\omega).
\]
Hence, $\beta$ is a group homomorphism. It is surjective by definition of $B^1(\sigma;\T)$. To see that $\beta$ is continuous, it suffices to show that it is sequentially continuous since the topology on $L^0(\Omega,\nu;\T)$ is metrizable by Lemma \ref{lem:top_on_L0(Omega,nu;D)}. Let $\varphi_n$ be a convergent sequence in $L^0(\Omega,\nu;\T)$ with limit $\varphi$. Then $\norm{\varphi-\varphi_n}{1}\rightarrow0$, by Lemma \ref{lem:top_on_L0(Omega,nu;D)}. For every $t\in \Gamma$, we have
\[
\norm{b_{\varphi}(t,\square)-b_{\varphi_n}(t,\square)}{1}\leq\norm{\varphi-\varphi_n}{1}+\norm{t.\varphi-t.\varphi_n}{1}= 2\norm{\varphi-\varphi_n}{1}\rightarrow0.
\]
In the last equality, we use that the measure is invariant for the action of $\Gamma$. Because $\Gamma$ is discrete, it follows that $b_{\varphi_n}\rightarrow b_{\varphi}$ in $B^1(\sigma;\T)$. Hence, $\beta$ is continuous. Finally, $\ker\beta$ consists of functions which are constant on the orbits of $\Gamma\act{\sigma}(\Omega,\nu)$. Since the action is ergodic, it follows that $\ker\beta$ is the subgroup of constant functions.
\end{proof}

\begin{thm}\label{thm:B1_closed_implies_strongly_ergodic}
Let $\Gamma$ be a countable discrete group, $(\Omega,\nu)$ a separable probability space and $\Gamma\act{\sigma}(\Omega,\nu)$ an ergodic p.m.p. action. If $B^1(\sigma;\T)$ is closed then $\sigma$ is strongly ergodic.
\end{thm}
\begin{proof}
Because $\Gamma$ is countable discrete and $\Gamma\act{\sigma}(\Omega,\nu)$ is a p.m.p. action on a separable probability space, $Z^1(\sigma;\T)$ is Polish by Proposition \ref{prop:Z^1_is_Polish}. Assume $B^1(\sigma;\T)$ is closed. Then it is a Polish group, by Proposition \ref{prop:Polish_permanence_subgroup}. Because the map $\beta:L^0(\Omega,\nu;\T)\rightarrow B^1(\sigma;\T)$ from Lemma \ref{lem:beta:L0->B1} is then a continuous, surjective homomorphism of Polish groups, we deduce from Theorem \ref{thm:open_mapping_Polish} that $\beta$ is open. Hence, for each $\varepsilon>0$, the set
\[
\beta\left(V_{\varepsilon}(1_{\Omega})\right)=\set{b_{\varphi}}{\varphi\in L^0(\Omega,\nu;\T)\mbox{ such that }\nu\left(\left\{\abs{1_{\Omega}-\varphi}\geq\varepsilon\right\}\right)<\varepsilon}
\]
is an open set in $B^1(\sigma;\T)$ containing $b_{1_{\Omega}}$. Employing Lemma \ref{lem:top_on_L0(Omega,nu;D)}, we can then, for each $\varepsilon>0$, find a finite subset $F_{\varepsilon}\subset\Gamma$ and a $\delta_{\varepsilon}>0$ such that the set
\[
\set{b\in B^1(\sigma;\T)}{\norm{b(t,\square)-1_{\Omega}}{1}<\delta_{\varepsilon}\mbox{ for every }t\in F_{\varepsilon}}
\]
is contained in $\beta\left(V_{\varepsilon}(1_{\Omega})\right)$. That is, if $b\in B^1(\sigma;\T)$ is any $1$-coboundary satisfying that $\norm{b(t,\square)-1_{\Omega}}{1}<\delta_{\varepsilon}$, for every $t\in F_{\varepsilon}$, then there is a $\psi\in L^0(\Omega,\nu;\T)$ such that $b=b_{\psi}$ and such that $\nu(\{\abs{1_{\Omega}-\psi}\geq\varepsilon\})<\varepsilon$.

Let $\seq{A}{n}$ be an asymptotically invariant sequence of measurable subsets of $\Omega$ and set, for each $n\in\N$, $\varphi_n=2\cdot 1_{A_n}-1_{\Omega}$ and $b_n=b_{\varphi_n}\in B^1(\sigma;\T)$. Then, for each $t\in\Gamma$,
\[
\norm{b_n(t,\square)-1_{\Omega}}{1}=\norm{\varphi_n-t.\varphi_n}{1}=2\norm{1_{A_n}-t.1_{A_n}}{1}=2\nu(A_n\triangle t.A_n)\rightarrow0.
\]
Let $\seq{N}{k}$ be a strictly increasing sequence in $\N$ such that $\norm{b_{N_k}(t,\square)-1_{\Omega}}{1}<\delta_{1/k}$, for every $k\in\N$. For each $k\in\N$, we can then find $\psi_k\in L^0(\Omega,\nu;\T)$ such that $b_{N_k}=b_{\psi_k}$ and such that $\nu(\{\abs{1_{\Omega}-\psi_k}\geq\tfrac{1}{k}\})<\tfrac{1}{k}$. Because $\ker\beta\cong\T$, by Lemma \ref{lem:beta:L0->B1}, we must have $\psi_k=e^{i\theta_k}\varphi_{N_k}$, for some $\theta_k\in[0,2\pi)$. Observe that
\[
(1_{\Omega}-e^{i\theta_k}\varphi_{N_k})(\omega)=\begin{cases}
        1-e^{i\theta_k} & \text{if } \omega \in A_{N_k}\\
        1+e^{i\theta_k} & \text{if } \omega \not\in A_{N_k}
    \end{cases}
\]
Hence, we can always choose $\theta_k$ such that $e^{i\theta_k}=\pm1$. Now, for $k\in\N$ where $e^{i\theta_k}=1$, we have $1_{\Omega}-\psi_k=2(1_{\Omega}-1_{A_{N_k}})$ and so
\[
\nu\left(\left\{\abs{1_{\Omega}-\psi_k}\geq\frac{1}{k}\right\}\right)=\nu\left(\left\{1_{\Omega}-1_{A_{N_k}}\geq\frac{1}{2k}\right\}\right)=\nu\left(A_{N_k}^{\complement}\right)\\
\]
Similarly, for $k\in\N$ where $e^{i\theta_k}=-1$, we have $1_{\Omega}-\psi_k=2\cdot 1_{A_{N_k}}$ and so
\[
\nu\left(\left\{\abs{1_{\Omega}-\psi_k}\geq\frac{1}{k}\right\}\right)=\nu\left(\left\{1_{A_{N_k}}\geq\frac{1}{2k}\right\}\right)=\nu\left(A_{N_k}\right)\\
\]
Hence, for each $k\in\N$, it is either the case that $\nu\left(A_{N_k}\right)<\tfrac{1}{k}$ or that $\nu\left(A_{N_k}^{\complement}\right)<\tfrac{1}{k}$. Then, for each $k\in\N$,
\[
\nu\left(A_{N_k}\right)\nu\left(A_{N_k}^{\complement}\right)\leq\min\left\{\nu\left(A_{N_k}\right),\nu\left(A_{N_k}^{\complement}\right)\right\}<\frac{1}{k}.
\]
It follows that
\[
\liminf_{n\in\N}\nu\left(A_n\right)\nu\left(A_n^{\complement}\right)=0.
\]
Hence, $\seq{A}{n}$ is trivially asymptotically $\Gamma$-invariant. Since $\seq{A}{n}$ was an arbitrary asymptotically $\Gamma$-invariant sequence, it follows that $\Gamma\act{\sigma}(\Omega,\nu)$ is strongly ergodic.
\end{proof}

The characterization of property $(\rm T)$ by Connes and Weiss (Theorem \ref{thm:ConnesWeiss}) together with Theorem \ref{thm:B1_closed_implies_strongly_ergodic} immediately implies the following corollary:
\begin{cor}\label{cor:(T)_and_closure_of_B1}
Let $\Gamma$ be a discrete group. If $B^1(\sigma;\T)$ is closed in $Z^1(\sigma;\T)$, for every ergodic p.m.p. action $\Gamma\act{\sigma}(\Omega,\nu)$, then $\Gamma$ has property $(\rm T)$.
\end{cor}

\begin{rk}\label{rk:(T)_and_closure_of_B1}
By Remark \ref{rk:ConnesWeiss}, it suffices in Corollary \ref{cor:(T)_and_closure_of_B1} to consider actions on diffuse standard probability spaces.
\end{rk}

\section{Proof of Theorem \ref{introthm:weakTLp_and_TLp}}\label{sec:weakTLp}
Let $\Gamma\act{\sigma}(\Omega,\nu)$ be an ergodic measure class preserving action of a discrete group on a separable $\sigma$-finite measure space. Let $1\leq p<\infty$ and $c\in Z^1(\sigma;\T)$ be given. For each $t\in\Gamma$ and $\xi\in L^p(\Omega,\nu)$, set
\[
\pi_{p,\sigma,c}(t)\xi=c_t\left(\frac{\mathrm{d}t.\nu}{\mathrm{d}\nu}\right)^{1/p}t.\xi.
\]
Then $\pi_{p,\sigma,c}$ is an isometric representation of $\Gamma$ on $L^p(\Omega,\nu)$. We refer to, e.g., \cite{Gardella2019AModernLook} for a modern review of group representations on $L^p$-spaces. In this section, we characterice when $\pi_{p,\sigma,c}$ has invariant vectors (see Proposition \ref{prop:pi(p,sigma,c)_invariant_vectors}). Moreover, we give a sufficient condition for the existence of almost invariant unit vectors in the setting where the action is measure preserving (see Proposition \ref{prop:pi(p,sigma,c)_almost_invariant_vectors}). Together with the connection between property $(\rm T)$ and the closure of the space of $1$-coboundaries shown in Corollary \ref{cor:(T)_and_closure_of_B1}, these insights allow us to show the equivalence of weak property $(\mathrm{T}_{L^p})$ and property $(\mathrm{T}_{L^p})$. This is Theorem \ref{introthm:weakTLp_and_TLp} in the introduction and Theorem \ref{thm:weakTLp_and_TLp} below.

\begin{prop}\label{prop:pi(p,sigma,c)_invariant_vectors}
Let $\Gamma$ be a discrete group and $\Gamma\act{\sigma}(\Omega,\nu)$ an ergodic measure class preserving action on a $\sigma$-finite measure space. For $c\in Z^1(\sigma;\T)$ and $1\leq p<\infty$, denote by $\pi_{p,\sigma,c}$ the associated representation of $\Gamma$ on $L^p(\Omega,\nu)$. Then $\pi_{p,\sigma,c}$ admits a non-zero invariant vector if and only if $c$ is a $1$-coboundary and $\nu$ is equivalent to a finite $\Gamma$-invariant measure.
\end{prop}
\begin{proof}
Suppose $\xi\in L^p(\Omega,\nu)$ is a non-zero invariant vector for $\pi_{p,\sigma,c}$. Then, for every $t\in\Gamma$,
\[
\xi=c_t\left(\frac{\mathrm{d}t.\nu}{\mathrm{d}\nu}\right)^{1/p}t.\xi\qquad\mbox{$\nu$-a.e.}\numberthis\label{eqn:inv_vector}
\]
For each $t\in\Gamma$, let $\Omega_0^t\subset\Omega$ be the subset where we have equality in (\ref{eqn:inv_vector}) and set $\Omega_0=\cap_{t\in\Gamma}\Omega_0^t$. Then $\Omega_0$ is measurable and co-null. Because $c_t$ takes values in $\T$, we see that
\[
\abs{\xi}=\left(\frac{\mathrm{d}t.\nu}{\mathrm{d}\nu}\right)^{1/p}\abs{t.\xi},\numberthis\label{eqn:inv_vector_abs}
\]
for every $t\in\Gamma$ and with equality on $\Omega_0$. Then, since the Radon-Nikodym derivative is strictly positive, it follows that the set $\{\xi=0\}\cap\Omega_0$ is $\Gamma$-invariant. Because $\Gamma\act{\sigma}(\Omega,\nu)$ is ergodic and $\xi$ is non-zero, this implies that $\nu(\{\xi=0\})=0$. We then get a $\nu$-almost everywhere uniquely defined measurable function $\varphi\in L^0(\Omega,\nu;\T)$ such that $\xi=\varphi\abs{\xi}$. Insert this into equation (\ref{eqn:inv_vector}) and apply equation (\ref{eqn:inv_vector_abs}) to obtain:
\[
\varphi\abs{\xi}=c_t\left(\frac{\mathrm{d}t.\nu}{\mathrm{d}\nu}\right)^{1/p}t.\varphi\abs{t.\xi}=c_t t.\varphi\abs{\xi},
\]
for every $t\in\Gamma$ and with equality on $\Omega_0$. Since, $\nu(\{\xi=0\})=0$, we deduce that
\[
c_t=\frac{\varphi}{t.\varphi}
\]
$\nu$-almost everywhere and for every $t\in\Gamma$. Hence, $c$ is a $1$-coboundary. Further, for every $t\in\Gamma$ and every measurable subset $B\subset\Omega$, we have
\begin{align*}
\integral{1_{t.B}\abs{\xi}^p}{\Omega}{\nu}&=\integral{t.1_{B}\frac{\mathrm{d}t.\nu}{\mathrm{d}\nu}\abs{t.\xi}^p}{\Omega}{\nu}\\
&=\integral{1_{B}\abs{\xi}^p}{\Omega}{\nu},
\end{align*}
where we have used equation (\ref{eqn:inv_vector_abs}) in the first equality and the change of variable formula in the second. This shows that the finite measure $\abs{\xi}^p\mathrm{d}\nu$ is $\Gamma$-invariant. Because $\nu(\{\xi=0\})=0$, we see that the measure given by $\abs{\xi}^p\mathrm{d}\nu$ is equivalent to $\nu$.

Conversely, suppose that $\nu$ is equivalent to a finite $\Gamma$-invariant measure $\mu$, and that $c$ is a $1$-coboundary. Write $c=b_{\varphi}$, for a $\varphi\in L^0(\Omega,\nu;\T)$. Since $\mu$ is finite, the Radon-Nikodym derivative $\tfrac{\mathrm{d}\mu}{\mathrm{d}\nu}$ is $\nu$-integrable. Set $\xi=\varphi\cdot(\tfrac{\mathrm{d}\mu}{\mathrm{d}\nu})^{1/p}$. Then $\xi$ is non-zero and lies in $L^p(\Omega,\nu)$. We have, for every measurable subset $B\subset\Omega$, the equality
\[
\integral{1_{t.B}\abs{\xi}^p}{\Omega}{\nu}=\mu(t.B)=\mu(B)=\integral{1_{B}\abs{\xi}^p}{\Omega}{\nu}.
\]
Further, for each $t\in\Gamma$,
\[
\integral{1_{B}\abs{\xi}^p}{\Omega}{\nu}=\integral{1_{t.B}\frac{\mathrm{d}t.\nu}{\mathrm{d}\nu}\abs{t.\xi}^p}{\Omega}{\nu},
\]
by the change of variable formula. Putting this together, we see that
\[
\integral{1_{t.B}\left(\abs{\xi}^p-\frac{\mathrm{d}t.\nu}{\mathrm{d}\nu}\abs{t.\xi}^p\right)}{\Omega}{\nu}=0,
\]
for every measurable subset $B\subset\Omega$ and every $t\in\Gamma$. In particular, this equality holds for the following measurable subsets of $\Omega$:
\[
t^{-1}.\left\{\abs{\xi}^p>\frac{\mathrm{d}t.\nu}{\mathrm{d}\nu}\abs{t.\xi}^p\right\}\qquad\mbox{and}\qquad
t^{-1}.\left\{\abs{\xi}^p<\frac{\mathrm{d}t.\nu}{\mathrm{d}\nu}\abs{t.\xi}^p\right\}.
\]
Hence,
\[
\abs{\xi}^p=\frac{\mathrm{d}t.\nu}{\mathrm{d}\nu}\abs{t.\xi}^p\qquad \mbox{$\nu$-a.e.}
\]
We deduce that 
\[
\pi_{p,\sigma,b_{\varphi}}(t)\xi=\frac{\varphi}{t.\varphi}\left(\frac{\mathrm{d}t.\nu}{\mathrm{d}\nu}\right)^{1/p}t.\xi=\varphi\left(\frac{\mathrm{d}t.\nu}{\mathrm{d}\nu}\right)^{1/p}\abs{t.\xi}=\varphi\abs{\xi}=\xi\qquad \mbox{$\nu$-a.e.}
\]
That is, $\xi$ is a non-zero invariant vector for $\pi_{p,\sigma,b_{\varphi}}$.
\end{proof}

\begin{prop}\label{prop:pi(p,sigma,c)_almost_invariant_vectors}
Let $\Gamma$ be a discrete group, let $(\Omega,\nu)$ be a probability space, let $\Gamma\act{\sigma}(\Omega,\nu)$ be a p.m.p. action and let $1\leq p<\infty$. If $c\in Z^1(\sigma;\T)$ is the limit of a net of $1$-coboundaries then $\pi_{p,\sigma,c}$ admits a net of almost invariant unit vectors.
\end{prop}
\begin{proof}
Let $(\varphi_i)_i$ be a net in $L^0(\Omega,\nu;\T)$. For each index $i$, write $b_i=b_{\varphi_i}$ for the associated $1$-coboundary. Assume that $(\varphi_i)_i$ is such that $c$ is the limit of $(b_i)_i$. For each index $i$, $\varphi_i$ lies in $L^p(\Omega,\nu)$ with unit norm since $(\Omega,\nu)$ is a probability space and $\varphi_i$ has everywhere modulus equal to $1$. We have, for all $t\in\Gamma$,
\[
\norm{\pi_{p,\sigma,c}(t)\varphi_i-\varphi_i}{p}=\norm{c_t\,t.\varphi_i-\varphi_i}{p}=\norm{c_t-b_i(t,\square)}{p}.
\]
Because $p$ is finite and because $c$ is the limit of $(b_i)_i$, Lemma \ref{lem:top_on_L0(Omega,nu;D)} implies that this converges to zero. Hence, $(\varphi_i)_i$ is a net of almost invariant unit vectors for $\pi_{p,\sigma,c}$.
\end{proof}

\begin{thm}\label{thm:weakTLp_and_TLp} 
A discrete group has property $(\mathrm{T}_{L^p})$ if and only if it has weak property $(\mathrm{T}_{L^p})$
\end{thm}
\begin{proof}
Let $\Gamma$ be a discrete group without property $(\mathrm{T}_{L^p})$. Then $\Gamma$ does not have property $(\rm T)$, by \cite[Theorem A]{BaderFurmanGelanderMonod}. Hence, by Corollary \ref{cor:(T)_and_closure_of_B1} and Remark \ref{rk:(T)_and_closure_of_B1}, there is an ergodic p.m.p. action $\Gamma\act{\sigma}(\Omega,\nu)$ on a separable probability space such that $B^1(\sigma;\T)$ is not closed. We can then find a $1$-cocycle $c\in Z^1(\sigma,\T)$ which is not a $1$-coboundary but which is the limit of a net of $1$-coboundaries. It follows from Proposition \ref{prop:pi(p,sigma,c)_invariant_vectors} and Proposition \ref{prop:pi(p,sigma,c)_almost_invariant_vectors} that $\pi_{p,\sigma,c}$ has almost invariant unit vectors but no non-zero invariant vector.
\end{proof}

\section{On the possibility of an easier proof}\label{sec:no_easier_proof}
The class $L^p$, for $1\leq p<\infty$ and $p\neq2$, is neither stable under quotients nor under complemented subspaces. Therefore, it does not meet either of the conditions of Proposition 2.20 in \cite{ElkiaerPooya2023} for the equivalence of weak property $(\mathrm{T}_{\mathcal{E}})$ and property $(\mathrm{T}_{\mathcal{E}})$. However, as mentioned in the introduction, these conditions are stronger than needed. In this section, we shall address the question if a proof of Theorem \ref{introthm:weakTLp_and_TLp} is possible via the ideas used in \cite{ElkiaerPooya2023}. Precisely, given an isometric $L^p$-representation $(\pi,L^p(\Omega,\nu))$, if one could guarantee the existence of another isometric $L^p$-representation $(\rho,L^p(\Omega',\mu))$ and a bounded isomorphism $\Phi:L^p(\Omega,\nu)/L^p(\Omega,\nu)^{\pi}\rightarrow L^p(\Omega',\mu)$ such that $\rho=\Phi\circ\pi\circ\Phi^{-1}$, the equivalence of weak property $(\mathrm{T}_{L^p})$ and property $(\mathrm{T}_{L^p})$ would follow directly. In Theorem \ref{thm:Lp0_is_not_an_Lp-space}, which is Theorem \ref{introthm:Lp0_is_not_an_Lp-space} in the introduction, we give an example of an isometric representation on an $L^p$-space where the quotient with the subspace of invariant vectors is not isometrically isomorphic to any $L^p$-space on a $\sigma$-finite measure space.\\

Let $\Gamma\act{\sigma}(\Omega,\nu)$ be an ergodic p.m.p. action and denote by $\pi_{p,\sigma}$ the associated representation on $L^p(\Omega,\nu)$ (with the trivial $1$-cocycle). Observe that $L^p(\Omega,\nu)^{\pi_{p,\sigma}}\cong \C 1_{\Omega}$, and so, the quotient $L^p(\Omega,\nu)/L^p(\Omega,\nu)^{\pi_{p,\sigma}}$ is isomorphic as a vector space to the $\Gamma$-invariant subspace of functions with mean zero:
\[
L^p_0(\Omega,\nu)=\set{f\in L^p(\Omega,\nu)}{\integral{f}{\Omega}{\nu}=0}.
\]
Denote by $\pi_{p,\sigma}^0$ the restriction of $\pi_{p,\sigma}$ to $L^p_0(\Omega,\nu)$ and by $\overline{\pi}_{p,\sigma}$ the representation on the quotient $L^p(\Omega,\nu)/L^p(\Omega,\nu)^{\pi_{p,\sigma}}$ coming from $\pi_{p,\sigma}$. A straight forward computation confirms that the isomorphism between $L^p(\Omega,\nu)/L^p(\Omega,\nu)^{\pi_{p,\sigma}}$ and $L^p_0(\Omega,\nu)$ is equivariant, i.e., that it intertwines $\pi_{p,\sigma}^0$ and $\overline{\pi}_{p,\sigma}$. However, it need not be isometric. Proposition \ref{prop:quotient_and_Lp0} below is well-known to experts.
\begin{prop}\label{prop:quotient_and_Lp0}
Let $\Gamma\act{\sigma}(\Omega,\nu)$ be an ergodic p.m.p. action. Let $1<p,p'<\infty$ be Hölder conjugates. The quotient $L^p(\Omega,\nu)/L^p(\Omega,\nu)^{\pi_{p,\sigma}}$ is equivariantly and isometrically isomorphic to the dual of $L^{p'}_0(\Omega,\nu)$.
\end{prop}
\begin{proof}
For $f\in L^p_0(\Omega,\nu)$ and $g\in L^{p'}_0(\Omega,\nu)$, set $\varphi_{f}(g)=\integral{fg}{\Omega}{\nu}$. Then $\varphi_{f}$ is a linear functional on $L^{p'}_0(\Omega,\nu)$ and it is straightforward to check that $f\mapsto\varphi_{f}$ defines a vector space isomorphism $\varphi:L^p_0(\Omega,\nu)\rightarrow L^{p'}_0(\Omega,\nu)'$. Denote by $(\pi_{p',\sigma}^0)'$ the dual representation of the restriction of $\pi_{p',\sigma}$ to $L^{p'}_0(\Omega,\nu)$. For each $f\in L^p_0(\Omega,\nu)$ and $g\in L^{p'}_0(\Omega,\nu)$ and for each $t\in \Gamma$,
\begin{align*}
(\pi_{p',\sigma}^0)'(t)\varphi_f(g)&=\varphi_f(\pi_{p,\sigma}^0(t^{-1})g)=\integral{g(t.\omega)f(\omega)}{\Omega}{\nu(\omega)}=\varphi_{\pi_{p,\sigma}(t)f}(g),
\end{align*}
where we have used invariance of $\nu$ in the last equality. This shows that $\varphi$ is equivariant. Precomposing $\varphi$ with the canonical equivariant isomorphism from $L^p(\Omega,\nu)/L^p(\Omega,\nu)^{\pi_{p,\sigma}}$ to $L^p_0(\Omega,\nu)$, we obtain an equivariant isomorphism from $L^p(\Omega,\nu)/L^p(\Omega,\nu)^{\pi_{p,\sigma}}$ to $L^{p'}_0(\Omega,\nu)'$ given by $[f]\mapsto\varphi_{f_0}$, where $f_0\in L^p_0(\Omega,\nu)$ is such that $f=f_0+c$, for some $c\in\C$. We claim that this map is isometric. Observe that if $f_0\in L^p_0(\Omega,\nu)$, $g_0\in L^p_0(\Omega,\nu)$ and $c\in\C$ then
\[
\integral{f_0(g_0+c)}{\Omega}{\nu}=\integral{f_0g_0}{\Omega}{\nu}=\integral{(f_0+c)g_0}{\Omega}{\nu}.\numberthis\label{eqn:quotient_and_Lp0_1}
\]
We apply the first of these equalities to see that
\[
\norm{\varphi_{f_0}}{}=\sup\set{\abs{\integral{f_0g_0}{\Omega}{\nu}}}{g\in L^{p'}_0(\Omega,\nu),\norm{g}{p'}\leq1}=\norm{f_0}{p},
\]
from which it follows that
\[
\norm{[f_0+c]}{}=\inf_{d\in\C}\norm{f_0+d}{p}\leq\norm{f_0}{p}=\norm{\varphi_{f_0}}{}.
\]
Conversely, utilizing the second equality in equation \ref{eqn:quotient_and_Lp0_1}, we see that, for each $d\in\C$,
\begin{align*}
\norm{f_0+d}{p}&=\sup\set{\abs{\integral{(f_0+d)g}{\Omega}{\nu}}}{g\in L^{p'}(\Omega,\nu),\norm{g}{p'}\leq1}\\
&\geq \sup\set{\abs{\integral{f_0g_0}{\Omega}{\nu}}}{g_0\in L^{p'}_0(\Omega,\nu),\norm{g}{p'}\leq1}.
\end{align*}
Hence,
\[
\norm{[f_0+c]}{}=\inf_{d\in\C}\norm{f_0+d}{p}\geq\norm{\varphi_{f_0}}{}.
\]
This proves the claim.
\end{proof}

We shall see in Theorem \ref{thm:Lp0_is_not_an_Lp-space} that, in many cases, $L^p_0(\Omega,\nu)$ \emph{is not} isometrically isomorphic to an $L^p$-space on a $\sigma$-finite measure space, and so, neither is its dual. The proof relies on the following extension theorem, which is Theorem 4 in \cite{HandbookBanach_21KoldobskyKonig}. We state it below without proof.

\begin{thm}[Extension Theorem]\label{thm:extension}
Let $1\leq p<\infty$, $p\not\in2\N$, and let $(\Omega_1,\nu_1)$ and $(\Omega_2,\nu_2)$ be probability spaces. Let $Y\subset L^p(\Omega_1,\nu_1)$ be a subspace containing $1_{\Omega_1}$ and denote by $\Sigma(Y)$ the smallest $\sigma$-algebra on $\Omega_1$ making all functions in $Y$ measurable. Let $\Phi:Y\rightarrow L^p(\Omega_2,\nu_2)$ be a linear isometry. There exists a linear isometry $\Phi':L^p(\Omega_1,\Sigma(Y),\nu_1)\rightarrow L^p(\Omega_2,\nu_2)$ such that $\Phi'\restrict{H}=\Phi$.
\end{thm}

Before proceeding to the proof of Theorem \ref{thm:Lp0_is_not_an_Lp-space}, we shall need two lemmas. Lemma \ref{lem:exist_equivalent_finite_measure} is well-known and can be verified with standard methods. We suspect that Lemma \ref{lem:rotating_subspace_of_balanced_fcts} is known, but we do not have a reference.

\begin{lem}\label{lem:exist_equivalent_finite_measure}
Let $(\Omega,\mu)$ be a $\sigma$-finite measure space and let $1\leq p<\infty$. There exists an equivalent probability measure $\nu$ on $\Omega$ such that $L^p(\Omega,\mu)$ is isometrically isomorphic to $L^p(\Omega,\nu)$.
\end{lem}

\begin{lem}\label{lem:rotating_subspace_of_balanced_fcts}
Let $(\Omega,\mathcal{B},\nu)$ be a diffuse standard probability space and let $1\leq p<\infty$. There exists an isometric isomorphism $U$ of $L^p(\Omega,\mathcal{B},\nu)$ such that
\begin{enumerate}[1.]
\item $1_{\Omega}\in U(L^p_0(\Omega,\mathcal{B},\nu))$,
\item The smallest $\sigma$-algebra making all functions in $U(L^p_0(\Omega,\mathcal{B},\nu))$ is $\mathcal{B}$.
\end{enumerate}
\end{lem}
\begin{proof}
It suffices to consider the case where $(\Omega,\mathcal{B},\nu)$ is the interval $[0,1]$ with the Lebesgue measure. Let $f_0\in L^p_0([0,1])$ be such that $\abs{f_0(\omega)}=1$, for all $\omega\in[0,1]$. Then the multiplication operator $M_{f_0}$ is an isometric isomorphism of $L^p([0,1])$. Because $L^p_0([0,1])$ is stable under complex conjugation and $f_0\overline{f_0}=1_{\Omega}$, we see that $1_{\Omega}\in M_{f_0}(L^p_0([0,1]))$. Further, for each pair $0\leq a<b\leq1$, set $f_{a,b}=1_{\left(a,(a+b)/2\right)}-1_{\left((a+b)/2,b\right)}$. Then $f_{a,b}$ has mean zero and $M_{f_0}f_{a,b}$ takes values in $\T$ on the interval $(a,b)$ and is zero elsewhere. Therefore, the open interval $(a,b)=(M_{f_0}f_{a,b})^{-1}(\T)$ is in the $\sigma$-algebra generated by $M_{f_0}(L^p_0([0,1]))$. Hence, $M_{f_0}(L^p_0([0,1]))$ generates the Borel $\sigma$-algebra on $[0,1]$.
\end{proof}

\begin{thm}\label{thm:Lp0_is_not_an_Lp-space}
Let $(\Omega,\nu)$ be a diffuse standard probability space and let $1\leq p<\infty$, $p\not\in2\N$. Then $L^p_0(\Omega,\nu)$ is not isometrically isomorphic to an $L^p$-space on a $\sigma$-finite measure space.
\end{thm}
\begin{proof}
By Lemma \ref{lem:exist_equivalent_finite_measure}, it suffices to show that $L^p_0(\Omega,\nu)$ is not isometrically isomorphic to an $L^p$-space on a probability space. Suppose for contradiction that there is a probability space $(\Omega',\nu')$ and a linear isometric isomorphism $\Phi:L^p_0(\Omega,\nu)\rightarrow L^p(\Omega',\nu')$. Let $U$ be a linear isometric isomorphism of $L^p(\Omega,\nu)$ as in Lemma \ref{lem:rotating_subspace_of_balanced_fcts}. We apply the Extension Theorem (Theorem \ref{thm:extension}) to the composition $\Phi\circ U^{-1}:U(L^p_0(\Omega,\nu))\rightarrow L^p(\Omega',\nu')$ to obtain a linear isometry $L^p(\Omega,\nu)\rightarrow L^p(\Omega',\nu')$ extending $\Phi$. But $\Phi\circ U^{-1}$ is already surjective as a map defined on $U(L^p_0(\Omega,\nu))$. Therefore, it cannot extend to an injective map out of a strictly larger space. Hence, the map $\Phi$ cannot exist.
\end{proof}

The example considered in this section of an ergodic p.m.p. action on a diffuse standard probability space $(\Omega,\nu)$ is for us the most important example. Indeed, if $L^p_0(\Omega,\nu)$ is bounded equivariantly isomorphic to an isometric representation on a space in $L^p$, the equivalence of weak property $(\mathrm{T}_{L^p})$ and property $(\mathrm{T}_{L^p})$ would follow from this together with the characterization of property $(\rm T)$ by Connes and Weiss (see Theorem \ref{thm:ConnesWeiss} and Remark \ref{rk:ConnesWeiss} in the preliminary section). We end this section by remarking that we have only partially refuted this proof strategy by showing that $L^p_0(\Omega,\nu)$ is not \emph{isometrically} isomorphic to a space in $L^p$. It remains an open question if it is possible to find an equivariant bounded isomorphism. We expect that the answer to this question is `No'.

\begin{question}
Let $\Gamma\act{\sigma}(\Omega,\nu)$ be an ergodic p.m.p. action of a discrete group $\Gamma$ on a diffuse standard probability space and let $1\leq p<\infty$, $p\neq2$. Does there exist a $\sigma$-finite measure space $(\Omega',\mu)$ and a bounded (not necessarily isometric) isomorphism $\Phi:L^p_0(\Omega,\nu)\rightarrow L^p(\Omega',\mu)$ such that $\Phi\circ \pi_{p,\sigma}^0 \circ\Phi^{-1}$ is an isometric representation of $\Gamma$ on $L^p(\Omega',\mu)$?
\end{question}

It is known that $L^p_0(\Omega,\nu)$ \emph{is} bounded isomorphic to an $L^p$-space. This can be shown via the \emph{decomposition method} (see \cite[page 14]{HandbookBanach_1JohnsonLindenstrauss}). However, the bounded isomorphism achieved in this way is not equivariant.

\bibliography{bib}
\bibliographystyle{plain}

\vspace{2em}
\begin{minipage}[t]{0.45\linewidth}
  \small    Emilie Mai Elki\ae{}r\\
            Department of Mathematics\\
            University of Oslo\\
            0851 Oslo, Norway\\
		    {\footnotesize elkiaer@math.uio.no}
\end{minipage}

\newpage
\listoffixmes

\end{document}